\documentclass[12pt]{amsart}
\usepackage{amssymb,amsmath,amsthm,mathrsfs,multirow,xcolor,framed,url}
\usepackage[pdftex,
         pdfauthor={Dandan Chen and Ziyin Zou},
         pdftitle={Combinatorial proof of identities involving partitions with distinct even parts and 4-regular partitions},
         pdfsubject={MATHEMATICS},
         pdfkeywords={integer partitions, q-series.},
         pdfproducer={Latex with hyperref},
         pdfcreator={pdflatex}]{hyperref}
\usepackage [latin1]{inputenc}
\newtheorem{theorem}{Theorem}[section]
\newtheorem{lemma}[theorem]{Lemma}

\theoremstyle{definition}
\newtheorem{definition}[theorem]{Definition}

\theoremstyle{remark}

\numberwithin{equation}{section}

\newcommand\nutwid{\overset {\text{\lower 3pt\hbox{$\sim$}}}\nu}

\allowdisplaybreaks

\newcommand\omycite[1]{}

\newcommand{\beqs}{\begin{equation*}}
\newcommand{\eeqs}{\end{equation*}}
\newcommand{\beq}{\begin{equation}}
\newcommand{\eeq}{\end{equation}}

\begin{document}
\title[Combinatorial proof of identities]{Combinatorial proof of identities involving partitions with distinct even parts and 4-regular partitions}


\author{Dandan Chen}
\address{Department of Mathematics, Shanghai University, People's Republic of China}
\address{Newtouch Center for Mathematics of Shanghai University, Shanghai, People's Republic of China}
\email{mathcdd@shu.edu.cn}
\author{Ziyin Zou}
\address{Department of Mathematics, Shanghai University, People's Republic of China}
\email{ziyinzou@126.com}



\date{}



\begin{abstract}
Recently Andrews and El Bachraoui proved identities relating certain restricted partitions into distinct even parts with restricted
4-regular partitions by the theory of basic hypergeometric series. They also posed a question regarding combinatorial proofs for these results. In this paper, we establish bijections to provide combinatorial proofs for these results.
\end{abstract}

\maketitle


\section{Introduction}

In 2009, Andrews \cite{Andrews-09} denoted the ped($n$) as the number of partitions of $n$ with distinct even parts and found that
\begin{align}\label{ped}
\sum_{n=0}^{\infty}\text{ped}(n)q^n=\frac{(-q^2;q^2)_{\infty}}{(q;q^2)_{\infty}}=\frac{(q^4;q^4)_{\infty}}{(q;q)_{\infty}},\quad \text{for $|q|<1$},
\end{align}
where the $q$-shifted factorial is defined by \cite{GR}
\begin{align*}
(a;q)_0=1, \; (a;q)_n=\prod_{j=0}^{n-1}(1-aq^j), \;
(a;q)_{\infty}=\prod_{j=0}^{\infty}(1-aq^j) \;\text{for $|q|<1$}.
\end{align*}
From \eqref{ped} we know that ped($n$) equals the number of partitions of $n$ whose parts are not divisible by $4$. These partitions and their arithmetic properties have been studied extensively in recent years, see for instance \cite{Andrews-Hirschhorn-Sellers,Chen,Lovejoy-12}.

Recently, Andrews and El Bachraoui \cite{Andrews-Bachraoui} used the formula in  \cite{Andrews-Subbarao-Vidyasagar}
\begin{align*}
\sum_{n=0}^{\infty}\frac{(a;q)_n}{(b;q)_n}q^n
=\frac{q(a;q)_\infty}{b(b;q)_\infty(1-{aq}/{b})}+\frac{1-q/b}{1-{aq}/b}, \;
 \left|q\right|<1,
\end{align*}
to prove more identities related to restricted partitions with distinct even parts and restricted $4$-regular partitions, where the $l$-regular partition is a partition whose summands are not divisible by $l$.

Motivated by the works of Andrews and El Bachraoui \cite{Andrews-Bachraoui}, we present a bijective proof of the following theorems.
\begin{theorem}\:\! \emph{(\!\cite{Andrews-Bachraoui}, Eq.(2.2))\label{thm-DE1}}.
For $|q|<1$, there holds
\begin{align}\label{theorem1}
 (1+q)\sum_{n\ge0}\frac{(-q^2;q^2)_n q^{2n+1}}{(q;q^2)_{n+1}}=\frac{(q^4;q^4)_{\infty}}{(q;q)_{\infty}}-1.
\end{align}
\end{theorem}

\begin{theorem}\:\! \emph{(\!\cite{Andrews-Bachraoui}, Eq.(2.4))\label{thm-DE2}}.
For $|q|<1$, there holds
\begin{align}\label{theorem3}
 (1+q^3)\sum_{n\ge0}\frac{(-q^2;q^2)_n q^{4n+2}}{(q;q^2)_{n+1}}=\frac{(q^4;q^4)_{\infty}}{(q^2;q)_{\infty}}-1.
\end{align}
\end{theorem}

\begin{theorem}\:\! \emph{(\!\cite{Andrews-Bachraoui}, Eq.(2.6))\label{thm-DE3}}.
For $|q|<1$, there holds
\begin{align}\label{theorem2}
 (1+q^3)\sum_{n\ge0}\frac{(-q^2;q^2)_n q^{2n+1}}{(q;q^2)_{n}}=\frac{q^2(q^4;q^4)_{\infty}}{(q;q)_{\infty}}-q^2+q.
\end{align}
\end{theorem}

In Section 2, we establish  bijections  to provide proofs of the combinatorial for  \eqref{theorem1}-- \eqref{theorem2},respectively.

\section{COMBINATORIAL PROOF OF THEOREM \ref{thm-DE1}--\ref{thm-DE3}}
A partition $\lambda$ of a positive integer $n$ is a finite non-increasing sequence of positive integers $\lambda=(\lambda_1,\lambda_2,\ldots,\lambda_k)$ such that $\sum\limits_{i=1}^{k}\lambda_{i}=n.$
\begin{definition} \emph{\cite{Andrews-Bachraoui}}.
Let DE1($n$) denote the number of partitions of $n$ in which no even part is repeated and the largest part is odd. Then
\begin{align}\label{DE1}
\sum_{n=0}^{\infty}\text{DE1}(n)q^n=\sum_{n=0}^{\infty}\frac{(-q^2;q^2)_nq^{2n+1}}{(q;q^2)_{n+1}}.
\end{align}
\end{definition}
\begin{definition} \emph{\cite{Andrews-Bachraoui}}.
Let DE2($n$) denote the number of partitions of $n$ in which no even part is repeated and the largest part is odd and appears at least twice. Then
\begin{align}\label{DE2}
\sum_{n=0}^{\infty}\text{DE2}(n)q^n=\sum_{n=0}^{\infty}\frac{(-q^2;q^2)_nq^{4n+2}}{(q;q^2)_{n+1}}.
\end{align}
\end{definition}
\begin{definition} \emph{\cite{Andrews-Bachraoui}}.
Let DE3($n$) denote the number of partitions of $n$ in which no even part is repeated and the largest part is odd and appears exactly once. Then
\begin{align}\label{DE3}
\sum_{n=0}^{\infty}\text{DE3}(n)q^n=\sum_{n=0}^{\infty}\frac{(-q^2;q^2)_nq^{2n+1}}{(q;q^2)_{n}}.
\end{align}
\end{definition}

We define $\pi_{ped}(n)$, $\pi_{DE1}(n)$, $\pi_{DE2}(n)$, $\pi_{DE3}(n)$ as the sets of ped partitions, DE1 partitions, DE2 partitions and DE3 partitions of $n$ respectively. We also let $\text{ped}_{n>1}(n)$ denote the number of partitions of $n$ in which each part is  strictly larger than $1$.

Now  we  provide  proofs of the combinatorial for  \eqref{theorem1}--\eqref{theorem2}, respectively.
\begin{proof}[Proof of Theorem \emph{1.1}]
\begin{lemma}\label{lem-DE1}
There exists a bijection $\phi_1$: $\pi_{ped}(n)\rightarrow \pi_{DE1}(n)\cup\pi_{DE1}(n-1)$ such that $\text{DE}1(n)+\text{DE}1(n-1)=\text{ped}(n)$, i.e. $\text{DE1}(n)+\text{DE1}(n-1)$ equals the number of $4$-regular partitions of $n$.
\end{lemma}
\begin{proof}
We define $\phi_1$ as follow:
for $n\ge1$, $\lambda=(\lambda_1,\lambda_2,\ldots,\lambda_k)$ $\in$ $\pi_{ped}$($n$), we divide $\lambda$ into two cases.

\textit{Case} 1:\; $\lambda_1$ is odd. Then $\phi_1(\lambda)=(\lambda_1,\lambda_2,\ldots,\lambda_k)$ $\in$ $\pi_{DE1}$($n$).

\textit{Case} 2:\; $\lambda_1$ is even. Then $\phi_1(\lambda)=(\lambda_1-1,\lambda_2,\ldots,\lambda_k)$ $\in$ $\pi_{DE1}$($n-1$), since the even parts of \text{ped}($n$) are distinct.

Obviously, $\phi_1$ is an injection. Now, we prove $\phi_1$ is a surjection. It suffices to construct a  map $\psi_1$ defined on $\pi_{DE1}(n)\cup\pi_{DE1}(n-1)$ such that for all $\lambda$ $\in$ $\pi_{ped}$($n$), we have $\psi_1(\phi_1(\lambda))=\lambda$  and for all $\mu$ $\in$ $\pi_{DE1}(n)\cup\pi_{DE1}(n-1)$, we have $\phi_1(\psi_1(\mu))=\mu$.

For $\mu=(\mu_1, \mu_2,\ldots,\mu_l)$ $\in$ $\pi_{DE1}(n)\cup\pi_{DE1}(n-1)$, we construct $\psi$ as follow:

\textit{Case} 1:\;$|\mu|=\mu_1+\cdots+\mu_l=n$, then $\psi_1(\mu)=(\mu_1, \mu_2,\ldots,\mu_l)\in \pi_{ped}(n)$.

\textit{Case} 2:\;$|\mu|=\mu_1+\cdots+\mu_l=n-1$, then $\psi_1(\mu)=\mu=(\mu_1+1, \mu_2,\ldots,\mu_l)\in\pi_{ped}(n)$.

By the constructions of $\phi_1$ and $\psi_1$, it is straightforward to check that $\psi_1(\phi_1(\lambda))=\lambda$ and $\phi_1(\psi_1(\mu))=\mu$, we complete the proof of Lemma \ref{lem-DE1}.
\end{proof}

Substituting the generating functions of \text{DE1}($n$) and \text{ped}($n$), \eqref{theorem1} holds.
\end{proof}

\begin{proof}[Proof of Theorem \emph{1.3}]
\begin{lemma}\label{lem-DE3}
For $n>0$, there exists a bijection $\phi_3: \pi_{ped}(n)\rightarrow \pi_{DE3}(n+2)\cup\pi_{DE3}(n-1)$ such that $\text{DE}3(n+2)+\text{DE}3(n-1)=\text{ped}(n)$, i.e. $\text{DE}3(n+2) + \text{DE}3(n-1)$ equals the number of $4$-regular partitions of $n$.
\end{lemma}
\begin{proof}
We construct $\phi_3$ as follow:
for $\lambda=(\lambda_1,\lambda_2,\ldots,\lambda_k)$ $\in$ $\pi_{ped}($n$)$, we divide $\lambda$ into two cases.

\textit{Case} 1:\; $\lambda_1$ is odd. Then $\phi_3(\lambda)=(\lambda_1+2,\lambda_2,\ldots,\lambda_{k}):=(\alpha_1,\ldots,\alpha_k)$ $\in$ $\pi_{DE3}(n+2)$, where $\alpha_1-\alpha_2\ge 2$.

\textit{Case} 2:\; $\lambda_1$ is even. Now we have two conditions,

(i) $\lambda_{2}=\lambda_1-1$. $\phi_3(\lambda)=(\lambda_2+2,\lambda_1,\lambda_3,\ldots,\lambda_k):=(\beta_1,\ldots,\beta_k)$ $\in$ $\pi_{DE3}$($n+2$) since $\lambda_{2}$ is odd, where $\beta_1-\beta_2=1$.

(ii) $\lambda_{2}<\lambda_1-1$. $\phi_3(\lambda)=(\lambda_1-1,\lambda_2,\ldots,\lambda_k )$ $\in$ $\pi_{DE3}$($n-1$) since $\lambda_{1}-1$ is unique.

It is easily to know $\phi_3$ is an injection. Now we construct a map $\psi_3$ on $\pi_{DE3}(n+2)\cup\pi_{DE3}(n-1)$ such that for all $\lambda$ $\in$ $\pi_{ped}$($n$), we have $\psi_3(\phi_3(\lambda))=\lambda$  and for all $\mu$ $\in$ $\pi_{DE3}(n+2)\cup\pi_{DE3}(n-1)$, we have $\phi_3(\psi_3(\mu))=\mu$. Then $\phi_3$ is a bijection can be proved.

We define $\psi_3$ as follow:
For $\mu=(\mu_1, \mu_2,\ldots,\mu_l)$ $\in$ $\pi_{DE3}(n+2)\cup\pi_{DE3}(n-1)$, we divede $\mu$ into three cases.

\textit{Case} 1:\;$|\mu|=\mu_1+\cdots+\mu_l=n+2$ and $\mu_1-\mu_2\ge 2$, then $\psi_3(\mu)=(\mu_1-2, \mu_2,\ldots,\mu_l)\in \pi_{ped}(n)$, since $\mu_1$ is odd.

\textit{Case} 2:\;$|\mu|=\mu_1+\cdots+\mu_l=n+2$ and $\mu_1-\mu_2=1$, then $\psi_3(\mu)=(\mu_2, \mu_1-2,\ldots,\mu_l)\in\pi_{ped}(n)$, since $\mu_1$ is odd and $\mu_2$ is even.

\textit{Case} 3:\;$|\mu|=\mu_1+\cdots+\mu_l=n-1$ , then $\psi_3(\mu)=(\mu_1+1, \mu_2,\ldots,\mu_l)\in\pi_{ped}(n)$, since $\mu_1$  is odd.

By the constructions of $\phi_3$ and $\psi_3$, it is straightforward to check that $\psi_3(\phi_3(\lambda))=\lambda$ and $\phi_3(\psi_3(\mu))=\mu$, we complete the proof of Lemma \ref{lem-DE3}.
\end{proof}

Substituting the generating functions of DE3($n$) and ped($n$), we complete the proof of  \eqref{theorem2}.
\end{proof}

\begin{proof}[Proof of Theorem \emph{1.2}]

\begin{lemma}

For $n>0$, $\text{DE}2(n) + \text{DE}2(n-3) = \text{ped}_{n>1}(n)$, i.e. DE$2$($n$)$+$DE$2$($n-3$) equals the number of 4-regular partitions of $n$ into parts each $>1$.

\end{lemma}

\begin{proof}
From the definitions of DE1($n$), DE2($n$) and DE3($n$), we deduce that $\text{DE}2(n)=\text{DE}1(n)-\text{DE}3(n)$. By Lemmas \ref{lem-DE1} and \ref{lem-DE3}, we have
\begin{align*}
\text{DE}1(n) + \text{DE}1(n-1) = \text{DE}3(n+2) + \text{DE}3(n-1).
\end{align*}
Then
\begin{align*}
&\text{DE}2(n)+\text{DE}2(n-3)\\
&=\text{DE}1(n)-\text{DE}3(n)+\text{DE}1(n-3)-\text{DE}3(n-3)\\
&=\text{DE}1(n)-\text{DE}1(n-2)\label{thm 2.6}\\
&=\text{DE}1(n)+\text{DE}1(n-1)-(\text{DE}1(n-1)+\text{DE}1(n-2))\\
&=\text{ped}(n)-\text{ped}(n-1)\\
&=\text{ped}_{n>1}(n).
\end{align*}
where \text{ped}($n-1$) equals the number of partitions of $n$ which $1$ appears.
\end{proof}
By the generating functions of DE2($n$) and ped($n$), we immediately obtain \eqref{theorem3}.
\end{proof}


\section{ACKNOWLEDGEMENTS}



\end{document}